\documentclass[12pt]{amsart}
\usepackage{amssymb,latexsym,amsmath,comment,tikz-cd}
\usepackage[mathscr]{eucal}
\usepackage{enumitem}
\usepackage{bbm}
\usepackage{youngtab}
\usepackage{xcolor}

\numberwithin{equation}{section}
\numberwithin{table}{section}
\numberwithin{figure}{section}

\def\cf{cf.~}

\newcommand{\hsp}[1]{{\hbox{\hspace{#1}}}}

\newcounter{letcnt} 

\def\a{\alpha}  

\def\d{\delta}

\def\z{\zeta}
\def\m{\mu}
\def\n{\nu}
\def\s{\sigma}

\def\w{\omega}

\def\tAd{\mathrm{Ad}} \def\tad{\mathrm{ad}}
 
\def\tAut{\mathrm{Aut}}

\def\bC{\mathbb C}

\def\td{\mathrm{d}}

 \def\tdim{\mathrm{dim}}
\def\cE{\mathcal E}

\def\tEnd{\mathrm{End}}

\def\cF{\mathcal F}

\def\cG{\mathcal G} 

\def\tGr{\mathrm{Gr}}
\def\fg{{\mathfrak{g}}}

 \def\sH{\mathscr{H}}
 
\def\tHom{\mathrm{Hom}}
 
\def\bh{\mathbf{h}}

 \def\tim{\mathrm{im}}

\def\fk{\mathfrak{k}}

 \def\tker{\mathrm{ker}}

\def\cM{\mathcal M}

 \def\cO{\mathcal O}
\def\tO{\mathrm{O}}

\def\bP{\mathbb P}

\def\fp{\mathfrak{p}}

 \def\cQ{\mathcal Q}

\def\bR{\mathbb R}

\def\trank{\mathrm{rank}}

 \def\tSO{\mathrm{SO}}
\def\tSp{\mathrm{Sp}}

 \def\tSym{\mathrm{Sym}}

\def\tSU{\mathrm{SU}}
 \def\tspan{\mathrm{span}}

\def\cT{\mathcal T}

 \def\cU{\mathcal U} \def\tU{\mathrm{U}}

 \def\cV{\mathcal V}

\def\sX{\mathscr{X}}

\def\half{\tfrac{1}{2}}

\def\tand{\quad\hbox{and}\quad}

\def\sb{{\hbox{\tiny{$\bullet$}}}}

\def\inj{\hookrightarrow}

\def\op{\oplus}
\def\ot{\otimes}

\def\tw{\hbox{\small $\bigwedge$}}

\newenvironment{a_list}
  {\begin{enumerate}[label=(\alph*),itemsep=3pt,leftmargin=25pt,listparindent=20pt]}
  {\end{enumerate}}

\newenvironment{i_list}
  {\begin{enumerate}[label=(\roman*),itemsep=3pt,leftmargin=25pt,listparindent=20pt]}
  {\end{enumerate}}

\newtheorem{mainthm}{Main Theorem}
\newtheorem{corollary}[equation]{Corollary}
\newtheorem{lemma}[equation]{Lemma}
\newtheorem{proposition}[equation]{Proposition}
\newtheorem{theorem}[equation]{Theorem}

\theoremstyle{definition}

\newtheorem*{boldQ*}{Question}
\newtheorem*{boldP*}{Problem}

\theoremstyle{definition}

\theoremstyle{remark}
\newtheorem*{assume*}{Assume}
\newtheorem*{answer*}{Answer}
\newtheorem{claim}[equation]{Claim}
\newtheorem*{claim*}{Claim}

\newtheorem*{definition*}{Definition}
\newtheorem{example}[equation]{Example}
\newtheorem*{example*}{Example}

\newtheorem*{hint*}{Hint}
\newtheorem*{notation*}{Notation}
\newtheorem{remark}[equation]{Remark}
\newtheorem*{remark*}{Remark}
\newtheorem*{remarks*}{Remarks}
\newtheorem*{fact*}{Fact}
\newtheorem*{emphL*}{Lemma}

\newtheorem*{emphQ*}{Question}
\newtheorem*{emphA*}{Answer}

\begin{document}
\title[Characterization of CY VHS]{Characterization of Calabi--Yau variations of Hodge structure over tube domains by characteristic forms}
\author[Robles]{Colleen Robles}
\email{robles@math.duke.edu}
\address{Mathematics Department, Duke University, Box 90320, Durham, NC  27708-0320} 
\thanks{Robles is partially supported by NSF grants DMS 1361120 and 1611939.}
\date{\today}
\begin{abstract}
Sheng and Zuo's characteristic forms are invariants of a variation of Hodge structure.  We show that they characterize Gross's canonical variations of Hodge structure of Calabi--Yau type over (Hermitian symmetric) tube domains.
\end{abstract}
{
}
\maketitle

\setcounter{tocdepth}{2}
\let\oldtocsection=\tocsection
\let\oldtocsubsection=\tocsubsection
\let\oldtocsubsubsection=\tocsubsubsection
\renewcommand{\tocsection}[2]{\hspace{0em}\oldtocsection{#1}{#2}}
\renewcommand{\tocsubsection}[2]{\hspace{3em}\oldtocsubsection{#1}{#2}}

\section{Introduction}

\subsection{The problem}

To every tube domain $\Omega = G/K$ Gross \cite{MR1258484} has associated a canonical (real) variation of Hodge structure (VHS)
\begin{equation}\label{E:canV}
  \begin{tikzcd}  \cV_\Omega \arrow[d] \\ \Omega \end{tikzcd}
\end{equation}
of Calabi--Yau (CY) type.  The construction of \eqref{E:canV} is representation theoretic, not geometric, in nature; in particular, the variation is \emph{not}, a priori, induced by a family 
\begin{equation}\label{E:X-S}
  \begin{tikzcd}  \sX \arrow[d,"\rho"] \\ S \end{tikzcd}
\end{equation}
of polarized, algebraic Calabi--Yau manifolds.  So an interesting problem is to construct such a family realizing \eqref{E:canV}.  By ``realize" we mean the following: let 
\begin{equation}\label{E:itau}
  \tau : \Omega \ \to \ D_\Omega
\end{equation}
be the period map associated with \eqref{E:canV}, and $\tilde\Phi_\rho : \tilde S \to D$ be the (lifted) period map associated with \eqref{E:X-S}; then we are asking for an identification $D \simeq D_\Omega$ with respect to which $\tilde\Phi_\rho(\tilde S)$ is an open subset of $\tau(\Omega)$.  

\begin{example} \label{eg:CY}
One may obtain a family of  $n$-folds by resolution of double covers of $\bP^n$ branched over $2n+2$ hyperplanes in general position.  When $n=1,2$, the associated VHS is a geometric realization of Gross's type $A$ canonical VHS over $\Omega = \tSU(n,n)/\mathrm{S}(\tU(n) \times \tU(n))$.  For $n=1$ this is the classical case of elliptic curves branched over fours points in $\bP^1$.  In the case $n=2$ this was proved by Matsumoto, Sasaki and Yoshida \cite{MR1136204}.  However, for $n\ge3$, the family does not realize Gross's type $A$ canonical VHS \cite{MR3004580, MR3303246}, \cf~Example \ref{eg:SZ}.
\end{example}  

A necessary condition for \eqref{E:X-S} to realize \eqref{E:canV} is that invariants associated to \eqref{E:canV} and \eqref{E:X-S} agree.  For example, $\tdim\,S = \tdim\,\Omega$, and the Hodge numbers $\bh_\rho$ and $\bh_\Omega$ must agree.  (Of course, the latter implies that we may identify $D$ with $D_\Omega$.)  These are discrete invariants.  Sheng and Zuo's characteristic forms \cite[\S3]{MR2657440} are infinitesimal, differential--geometric invariants associated with holomorphic, horizontal maps (such as $\tau$ and $\tilde\Phi_\rho$).  In particular, the characteristic forms will necessarily agree when \eqref{E:X-S} realizes \eqref{E:canV}.  

\begin{example} \label{eg:SZ}
When $n\ge 3$ the family of Calabi--Yau's in Example \ref{eg:CY} does \emph{not} realize Gross's type $A$ canonical VHS over $\Omega = \tSU(n,n)/S(\tU(n) \times \tU(n))$.  (However, the two discrete invariants above \emph{do} agree.)  This was proved by Gerkmann, Sheng, van Straten and Zuo \cite{MR3004580} in the $n=3$ case, and their argument was extended to $n\ge 3$ by Sheng, Xu and Zuo \cite{MR3303246}.  \emph{The crux of the argument is to show that the second characteristic forms do not agree.}  (In fact, their zero loci are not of the same dimension if $n\ge3$.)\footnote{A similar argument was used by Sasaki, Yamaguchi and Yoshida \cite{MR1476250} to disprove a related conjecture on the projective solution of the system of hypergeometric equations associated with the hyperplane configurations.}
\end{example}

The purpose of this paper is to show that agreement of the characteristic forms is both necessary and \emph{sufficient} for \eqref{E:X-S} to realize \eqref{E:canV}.  We will consider a more general situation, replacing the period map $\tilde\Phi_\rho : \tilde S \to D\simeq D_\Omega$ with an arbitrary horizontal, holomorphic map $f : M \to \check D_\Omega$ into the compact dual, and asking when $f$ realizes \eqref{E:canV}.  The first main result is stated precisely in Theorem \ref{T:main1}.  To state the informal version, we first recall that Gross's canonical VHS is given by a real representation 
\begin{equation} \label{E:iG}
  G \ \to \ \tAut(U,Q) \ 
  := \ \{ g \in \tAut(U) \ | \ Q(gu,gv) = Q(u,v) \,,\ \forall \ u,v \in U \}\,;
\end{equation}
the period domain $D_\Omega$ parameterizes (real) $Q$--polarized Hodge structures on $U$ of Calabi--Yau type; and the period map \eqref{E:itau} extends to a $G_\bC$--equivariant map $\tau : \check \Omega \to \check D_\Omega$ between the compact duals.

\begin{mainthm}[Informal statement of Theorem \ref{T:main1}] \label{T:mt1}
If the characteristic forms of $f$ and $\tau$ are isomorphic, then there exists $g \in \tAut(U_\bC)$ so that $g\circ f(M)$ is an open subset of $\tau(\check\Omega)$.
\end{mainthm}

\noindent Characteristic forms are defined in \S\ref{S:C}. The statement of Theorem \ref{T:main1} is a bit stronger than the above: in fact, it suffices to check that the characteristic forms of $f$ are isomorphic to those of $\tau$ at a single point $x \in M$, so long as the integer-valued differential invariants (\S\ref{S:Cisom}) associated with $f$ are constant in a neighborhood of $x$.  Theorem \ref{T:main1} is a consequence of: (i) an identification of the characteristic forms of Gross's \eqref{E:canV} with the fundamental forms of the minimal homogeneous embedding $\s: \check\Omega \inj \bP U_\bC$ (Proposition \ref{P:C=F}), and (ii) Hwang and Yamaguchi's characterization \cite{MR2030098} of compact Hermitian symmetric spaces by their fundamental forms.

Main Theorem \ref{T:mt1} characterizes horizontal maps realizing Gross's canonical VHS modulo the full linear automorphism group $\tAut(U_\bC)$.  It is natural to ask if we can characterize the horizontal maps realizing Gross's VHS up to the (smaller)  group $\tAut(U_\bC,Q)$ preserving the polarization --- these groups are the natural symmetry groups of Hodge theory.  (Note that $\tAut(U_\bC,Q)$ is the automorphism group of $\check D_\Omega$, the full $\tAut(U_\bC)$ does not preserve the compact dual.)  The second main result does exactly this.  This congruence requires a more refined notion of agreement of the characteristic forms than the isomorphism of Main Theorem \ref{T:mt1}; the precise statement is given in Theorem \ref{T:main2}.  The refinement is encoded by the condition that a certain vector-valued differential form $\eta$ vanishes on a frame bundle $\cE_f \to M$ (\cf~Remark \ref{R:eta=0}(b)).  Informally, one begins with a frame bundle $\cE_Q \to \check D_\Omega$ with fibre over $(F^p) \in \check D_\Omega$ consisting of all bases $\{e_0 , \ldots , e_d\}$ of $U_\bC$ such that $Q(e_j,e_k) = \d_{j+k}^{d}$ and $F^p = \tspan\{ e_0 , \ldots , e_{d^p} \}$.  The bundle $\cE_Q$ is isomorphic to the Lie group $\tAut(U_\bC,Q)$, and so inherits the left-invariant, Maurer-Cartan form $\theta$ which takes values in the Lie algebra
\[
  \tEnd(U_\bC,Q) \ := \ \{ X \in \tEnd(U_\bC) \ | \ Q(Xu,v) + Q(u,Xv) = 0 \,,\
  \forall \ u,v \in U_\bC\}
\]
of $\tAut(U_\bC,Q)$.  There is a $G_\bC$--module decomposition $\tEnd(U_\bC,Q) = \fg_\bC \op \fg^\perp_- \op \fg^\perp_{\ge0}$; let $\eta = \theta_{\fg^\perp_-}$ be the component of $\theta$ taking value in $\fg_-^\perp$.     

\begin{mainthm}[Informal statement of Theorem \ref{T:main2}] \label{T:mt2}
Let $f : M \to \check D_\Omega$ be a holomorphic, horizontal map.  There exists $g \in \tAut(U_\bC,Q)$ so that $g \circ f(M)$ is an open subset of $\tau(\check\Omega)\subset \check D_\Omega$ if and only if $\eta$ vanishes on the pull-back $\cE_f := f^*\cE_Q \to M$.
\end{mainthm}

\noindent Roughly speaking, $\eta$ vanishes on $\cE_f$ if and only if the coefficients of the fundamental forms of $f$ agree with those of Gross's canonical CY-VHS when expressed in terms of bases $\mathbf{e} \in \cE_Q$ (Remark \ref{R:eta=0}).   Main Theorem \ref{T:mt2} is reminiscent of Green--Griffiths--Kerr's characterization of nondegenerate complex variations of quintic mirror Hodge structures by the Yukawa coupling (another differential invariant associated to a VHS) \cite[\S IV]{MR2457736}.   Both Main Theorems \ref{T:mt1} and \ref{T:mt2}, and the Green--Griffiths--Kerr characterization, are solutions to equivalence problems in the sense of \'E.~Cartan.  And from that point of view, the formulation of Main Theorem \ref{T:mt2} is standard in that it characterizes equivalence by the vanishing of a certain form on a frame bundle over $M$.

The proof of Theorem \ref{T:main2} is established by a minor modification of the arguments employed in \cite{MR3004278} (which are similar to those of \cite{MR2030098}), and is in the spirit of Cartan's approach to equivalence problems via the method of moving frames.

\begin{remark}
Mao and Sheng \cite[\S2]{MR2657440} extended Gross's construction of the canonical \emph{real} CY-VHS over a \emph{tube domain} to a canonical \emph{complex} CY-VHS over a \emph{bounded symmetric domain}.  The analogs of Theorems \ref{T:main1} and \ref{T:main2} hold for the Mao--Sheng CY-VHS as well.  Specifically, the definition of the characteristic forms holds for arbitrary (not necessarily real) VHS; and the arguments establishing the theorems do not make use of the hypotheses that the bounded symmetric domain $\Omega$ is of tube type or that the VHS is real.  As indicated by the proofs of Theorems \ref{T:main1} and \ref{T:main2}, the point at which some care must be taken is when considering the case that $\check\Omega$ is either a projective space or a quadric hypersurface.  If $\check\Omega$ is not of tube type, then it can not be a quadric hypersurface.  If $\check\Omega$ is a projective space, then $\check\Omega = \check D_\Omega$, and the theorems are trivial. 
\end{remark}

\subsection{Notation} \label{S:prelim}

Throughout $V$ will denote a real vector space, and $V_\bC$ the complexification.  All Hodge structures are assumed to be effective; that is, the Hodge numbers $h^{p,q}$ vanish if either $p$ or $q$ is negative.  Throughout $\check D$ will denote the compact dual of a period domain $D$ parameterizing effective, polarized Hodge structures of weight $n$ on $V$.  Here $D$ and $V$ are arbitrary; we will reserve $D_\Omega$ and $U$ for the period domain and vector space specific to Gross's canonical variation of Hodge structure.  We will let $Q$ denote the polarization on both $V$ and $U$, as which is meant will be clear from context.

\section{Characteristic forms} \label{S:C}

\subsection{Horizontality}

Let 
\begin{equation}\label{E:cF}
  \cF^n \ \subset \ \cF^{n-1} \ \subset \cdots \subset \ \cF^1 \ 
  \subset \cF^0
\end{equation}
denote the canonical filtration of the trivial bundle $\cF^0 = \check D \times V_\bC$ over $\check D$.  Given a holomorphic map $f : M \to \check D$, let 
\[
  \cF^p_f \ := \ f^*\cF^p
\]
denote the pull-back of the Hodge bundles to $M$.  We say that $f$ is \emph{horizontal} if it satisfies the \emph{infinitesimal period relation} (IPR) 
\begin{equation}\label{E:IPR}
  \td \cF^p_f \ \subset \ \cF^{p-1}_f \ot \Omega^1_M \,.
\end{equation}

\begin{example}
The lifted period map $\tilde \Phi : \tilde S\to D$ arising from a family $\sX \to S$ of polarized, algebraic manifolds is a horizontal, holomorphic map \cite{MR0229641, MR0233825}.
\end{example}

\subsection{Definition}

Given a horizontal map $f : M \to \check D$, the IPR \eqref{E:IPR} yields a vector bundle map
\[
  \gamma_f : T M \ \to \ \tHom(\cF^n_f,\cF^{n-1}_f/\cF^n_f) \,;
\]
sending $\xi \in T_xM$ to the linear map $\gamma_{f,x}(\xi) \in \tHom(\cF^n_{f,x},\cF^{n-1}_{f,x}/\cF^n_{f,x})$ defined as follows.  Fix a locally defined holomorphic vector field $X$ on $M$ extending $\xi = X_x$.  Given any $v_0 \in \cF^n_{f,x}$, let $v$ be a local section of $\cF^n_f$ defined in a neighborhood of $x$ and with $v(x) = v_0$.  Then 
\[
  \gamma_f(\xi)(v_0) \ := \ \left.X(v)\right|_x \ \hbox{mod} \ \cF^n_{f,x}
\]
yields a well-defined map $\gamma_f(\xi) \in \tHom(\cF^n_f,\cF^{n-1}_f/\cF^n_f)$.
More generally there is a vector bundle map
\[
  \gamma_f^k : \tSym^k T M \ \to \ 
  \tHom(\cF^n_f , \cF^{n-k}_f/\cF^{n-k+1}_f) 
\]
defined as follows.  Given $\xi_1 , \ldots , \xi_k \in T_xM$, let $X_1 , \ldots , X_k$ be locally defined holomorphic vector fields extending the $\xi_j = X_{j,x}$.  Given $v_0$ and $v$ as above, define
\begin{equation}\label{E:gammak}
  \gamma_f^k(\xi_1, \ldots ,\xi_k)(v_0) \ : = \ 
  \left.X_1\cdots X_k(v)\right|_x \ \hbox{mod} \ \cF^{n-k+1}_{f,x} \,.
\end{equation}
It is straightforward to confirm that $\gamma_f^k$ is well-defined.  This bundle map is the \emph{$k$-th characteristic form} of $f:M\to \check D$.  Let $\mathbf{C}^k_f \subset \tSym^k T^* M$ denote the image of the dual map.  In a mild abuse of terminology we will also call $\mathbf{C}^k_f$ the \emph{$k$--th characteristic forms} of $f:M \to \check D$.

\subsection{Isomorphism} \label{S:Cisom}
Given two horizontal maps $f : M \to \check D$ and $f' : M' \to \check D$, we say that the characteristic forms of $f$ at $x$ are \emph{isomorphic} to those of $f'$ at $x'$ if there exists a linear isomorphism $\lambda : T_xM \to T_{x'}M'$ such that the induced linear map $\lambda^k : \tSym^k(T^*_{x'}M') \to \tSym^k (T^*_xM)$ identifies $\mathbf{C}^k_{f',x'}$ with $\mathbf{C}^k_{f,x}$, for all $k \ge 0$.

Each $\mathbf{C}^k_{f,x}$ is a vector subspace of $\tSym^k T^*_xM$, and 
\[
  c_{f,x}^k \ := \ \tdim_\bC\,\mathbf{C}^k_{f,x} 
  \ \le \ \tdim\,\cF^{n-k}_{f,x}/\cF^{n-k+1}_{f,x}
\]
is an example of an ``integer--valued differential invariant of $f: M \to \check D$ at $x$.''  Let
\[
  \mathbf{C}_{f,x} \ := \ 
  \bigoplus_{k\ge0} \mathbf{C}^k_{f,x}\ \subset \ 
  \bigoplus_{k\ge0} \tSym^k T^*_xM \ =: \ 
  \tSym\,T^*_xM \,,
\]
and set $c_{f,x} := \tdim_\bC\,\mathbf{C}_{f,x} = \sum_{k\ge0} c^k_{f,x}$.  Regard $\mathbf{C}_{f,x}$ as an element of the Grassmannian $\tGr(c_{f,x},\tSym\,T^*_xM)$.  Note that $\tAut(T_xM)$ acts on this Grassmannian.  By \emph{integer--valued differential invariant of $f:M \to \check D$ at $x$} we mean the value at $\mathbf{C}_{f,x}$ of any $\tAut(T_xM)$--invariant integer-valued function on $\tGr(c_{f,x},\tSym\,T^*_xM)$.

A necessary condition for two characteristic forms $\mathbf{C}_{f,x}$ and $\mathbf{C}_{f',x'}$ to be isomorphic is that the integer-valued differential invariants at $x$ and $x'$, respectively, agree.

\section{Gross's canonical CY-VHS} \label{S:gross}

\subsection{Maps of Calabi-Yau type} \label{S:CY}

A period domain $D$ parameterizing effective polarized Hodge structures of weight $n$ is \emph{of Calabi--Yau type} (CY) if $h^{n,0} = 1$.  In this case we also say that the compact dual $\check D$ is of Calabi--Yau type.

A holomorphic, horizontal map $f : M \to \check D$ is of \emph{Calabi--Yau} (CY) \emph{type} if $\check D$ is CY and $\gamma_{f,x} : T_xM \to \tHom(\cF^n_{f,x} , \cF^{n-1}_{f,x}/\cF^n_{f,x})$ is a linear isomorphism for all $x \in M$.

\begin{remark} \label{R:C1}
In particular, if $f : M \to \check D$ and $f': M' \to \check D$ are CY, then the first characteristic forms $\mathbf{C}^1_{f,x}$ and $\mathbf{C}^1_{f',x'}$ are always isomorphic, for any $x \in M$ and $x'\in M'$.
\end{remark}

The condition that $h^{n,0} = \trank_\bC\,\cF^n = 1$ implies that there is an map
\[
  \pi : \check D \ \to \ \bP V_\bC
\]
sending $\phi \in D$ to $\cF^n_\phi \in \bP V_\bC$.

\subsection{Definition} 

We briefly recall Gross's canonical CY-VHS over a tube domain $\Omega = G/K$ \cite{MR1258484}.  Up to $G$--module isomorphism, there is a unique real representation 
\begin{equation}\label{E:G-Aut(U)}
  G \ \to \ \tAut(U)
\end{equation}
with the following properties:
\begin{i_list}
\item The complexification $U_\bC$ is an irreducible $G$--module.
\item The maximal compact subgroup $K \subset G$ is the stabilizer of a highest weight line $\ell \subset U_\bC$.  In particular, if $P \subset G_\bC$ is the stabilizer of $\ell$, then $K = G \cap P$, and the map $gP \mapsto g \cdot \ell \in \bP U_\bC$ is a $G_\bC$--equivariant homogeneous embedding 
\begin{equation}\label{E:sigma'}
   \sigma : \check \Omega \ \inj \ \bP U_\bC
\end{equation}
of the compact dual $\check \Omega = G_\bC/P$ of $\Omega$.
\item The dimension of $U$ is minimal amongst all $G$--modules with the two properties above.
\end{i_list}

\noindent
The maximal compact subgroup $K$ is the centralizer of a circle $\varphi : S^1 \to G$ (a homomorphism of $\bR$--algebraic groups).  The representation $U_\bC$ decomposes as a direct sum 
\begin{subequations} \label{SE:Uhd}
\begin{equation}
  U_\bC \ = \ \bigoplus_{p+q=n} U^{p,q}
\end{equation}
of $\varphi$--eigenspaces
\begin{equation}
  U^{p,q} \ := \ \{ u \in U_\bC \ | \ \varphi(z) u = z^{p-q} u \} \,.
\end{equation}
\end{subequations}
This is a Hodge decomposition, and there exists a $G$--invariant polarization $Q$ of the Hodge structure; in particular, the representation \eqref{E:G-Aut(U)} takes values in $\tAut(U,Q)$:
\begin{equation}\label{E:G-Aut(U,Q)}
  G \ \to \ \tAut(U,Q) \,.\footnote{In the terminology of \cite{MR2918237}, the triple $(U_\bC,\varphi,Q)$ defines a Hodge representation which realizes the tube domain $\Omega$ as a Mumford--Tate domain.}  
\end{equation}
Each subset $U^{p,q}$ is $K$--invariant, and so defines a $G$--homogeneous bundle $\cU^{p,q}$ over $\Omega$.  The resulting decomposition
\begin{equation}\label{E:GcVHS}
  \Omega \times U_\bC \ = \ \bigoplus \cU^{p,q}
\end{equation}
of the trivial bundle over $\Omega$ is \emph{Gross's canonical VHS over $\Omega$} \cite{MR1258484}.

\begin{example}
In the case that $\Omega$ is irreducible, Gross's canonical CY-VHS is one of the following six:
\begin{a_list}
\item
For $G = \tU(n, n) = \tAut(\bC^{2n},\sH)$, we have $U_\bC = \tw^n\bC^{2n}$ and $\check \Omega = \tGr(n,\bC^{2n})$.  If $\bC^{2n} = A \op B$ is the $\varphi$--eigenspace decomposition, then $n = \tdim\,A = \tdim\,B$ and the Hermitian form $\sH$ restricts to a definite form on both $A$ and $B$.  The Hodge decomposition is given by $U^{p,q} \simeq (\tw^pA) \ot (\tw^q B)$.
\item
For $G = \tO(2,k) = \tAut(\bR^{2+k},Q)$, we have $U_\bC = \bC^{2+k}$ and $\Omega$ is the period domain parameterizing $Q$--polarized Hodge structures on $U = \bR^{2+k}$ with $\bh = (1,k,1)$, so that $\check\Omega$ is the quadric hypersurface $\{Q=0\} \subset \bP^{k+1}$.
\item 
For $G = \tSp(2g,\bR) = \tAut(\bR^{2g},Q)$, we have $U_\bC = \tw^g\bC^{2g}$ and $\Omega$ is the period domain parameterizing $Q$--polarized Hodge structures on $\bC^{2g}$ with $\bh = (g,g)$, so that $\check \Omega$ is the Lagrangian grassmannian of $Q$--isotropic $g$--planes in $\bC^{2g}$.  Given one such Hodge decomposition $\bC^{2n} = A \op B$, the corresponding Hodge structure on $U$ is given by $U^{p,q} = (\tw^pA) \op (\tw^q B)$.
\item 
For $G = \tSO^*(2n)$, $U_\bC$ is a Spinor representation, and the summands of the Hodge decomposition are $U^{p,q} \simeq \tw^{2p}\bC^{2n}$.
\item 
If $G$ is the exceptional simple real Lie group of rank $7$ with maximal compact subgroup $K = U(1) \times_{\mu_3} E_6$, then the Hodge decomposition is $U_\bC \simeq \bC \op \bC^{27} \op (\bC^{27})^* \op \bC$.
\end{a_list}
\end{example}

\begin{lemma}[Gross {\cite{MR1258484}}] \label{L:gross}
Gross's canonical VHS \eqref{E:GcVHS} is of Calabi--Yau type \emph{(\S\ref{S:CY})}.
\end{lemma}

\noindent The lemma follows from the well-understood representation theory associated with \eqref{E:sigma'} and \eqref{E:G-Aut(U,Q)}.  We briefly review the argument below as a means of recalling those representation theoretic properties that will later be useful.  (See \cite{MR1258484} for details.)  

Let 
\[
  \varphi \ \in \ D_\Omega
\]
denote the Hodge structure given by \eqref{SE:Uhd}.  The map
\begin{equation}\label{E:tau}
   \tau : \check \Omega \ \inj \ \check D_\Omega
\end{equation}
sending $gP \mapsto g \cdot \varphi$ is a $G_\bC$--equivariant homogeneous embedding of the compact dual $\check \Omega = G_\bC/P$.  The restriction of $\tau$ to $\Omega$ is the period map associated to Gross's canonical CY-VHS.  The precise statement of Main Theorem \ref{T:mt1} is 

\begin{theorem} \label{T:main1}
Let $f : M \inj \check D_\Omega$ be any CY map \emph{(\S\ref{S:CY})}, and let $x \in M$ be a point admitting a neighborhood in which all integer-valued differential invariants of $f$ are constant \emph{(\S\ref{S:Cisom})}.  If the characteristic forms of $f$ at $x$ are isomorphic to the characteristic forms of $\tau : \check\Omega \inj \check D_\Omega$ at $o \in \Omega$ in the sense of \emph{\S\ref{S:Cisom}}, then there exists $g \in \tAut(U_\bC)$ so that $g \circ f(M)$ is an open subset of $\tau(\check \Omega)$.
\end{theorem}

\noindent The theorem is proved in \S\ref{S:CvF}.

\begin{remark}  \label{R:main1}
To see how Main Theorem \ref{T:mt1} follows from Theorem \ref{T:main1} we make precise the hypothesis that ``the characteristic forms of $f$ and $\tau$ are isomorphic'':  by this, we mean that there exists a local biholomorphism $i : M \to \check\Omega$ so that the characteristic forms of $f$ at $x \in M$ are isomorphic to those of $\tau$ at $i(x)$ for all $x\in M$ (\cf\S\ref{S:Cisom}).  (Equivalently, since $\check\Omega$ is homogeneous, the characteristic forms of $f$ at $x \in M$ are isomorphic to those of $\tau$ at $o$ for all $x \in M$.)  Given this definition, it is clear that the hypotheses of Main Theorem \ref{T:mt1} imply those of Theorem \ref{T:main1}.
\end{remark}

\begin{proof}[Proof of Lemma \ref{L:gross}]
Let
\[
  \bh_\Omega \ = \ (h^{p,q}_\Omega = \tdim_\bC\,U^{p,q})
\]
denote the Hodge numbers, and let $D_\Omega$ denote the period domain parameterizing $Q$--polarized Hodge structures on $U$ with Hodge numbers $\bh_\Omega$.  The weight $n$ of the Hodge structure is the rank of $\Omega$, and the highest weight line stabilized by $K$ is 
\begin{equation}\label{E:ell}
  \ell \ = \ U^{n,0} \,.
\end{equation}
In particular, 
\begin{equation}\label{E:h=1}
  h^{n,0} \ = \ 1 \,.
\end{equation}
Let 
\[
  0 \ \subset \ \cF^n_\Omega \ \subset \ \cF^{n-1}_\Omega \ \subset \,\cdots\,
  \subset \ \cF^1_\Omega \ \subset \ \cF^0_\Omega 
\]
denote the canonical filtration \eqref{E:cF} of the trivial bundle $\cF^0_\Omega = \check D_\Omega \times U_\bC$ over $\check D_\Omega$.  Then
\[
  \left.\cF^p_\Omega\right|_{\tau(\Omega)} \ = \ 
  \bigoplus_{r\ge p} \cU^{r,n-r} \,.
\]

We will identify 
\[
  o \ = \ K/K \in \Omega \ = \ G/K
\]
with $P/P \in \check\Omega = G_\bC/P$.  Note that 
\[
  \varphi \ = \ \tau(o) \,.
\]
The weight zero Hodge decomposition 
\begin{equation}\label{E:ghd}
  \fg_\bC \ = \ \fg^{1,-1}_\varphi \,\op\, \fg^{0,0}_\varphi\,\op\,
  \fg^{-1,1}_\varphi
\end{equation}
induced by $\varphi$ has the property that $\fp = \fg^{1,-1}_\varphi \op \fg^{0,0}_\varphi$ and $\fk_\bC = \fg^{0,0}_\varphi$ are the Lie algebras of $P$ and $K_\bC$, respectively.  Consequently, the holomorphic tangent space is given by
\begin{equation}\label{E:TO}
  T_o\Omega \ = \ T_o\check\Omega \ = \ \fg_\bC/\fp \ \simeq \ 
  \fg^{-1,1}_\varphi \,.
\end{equation}
Regarding $\fg^{-1,1}_\varphi$ as a subspace of $\tEnd(U_\bC,Q)$ we have 
\begin{equation}\label{E:ontoU}
  U^{p-1,q+1} \ = \ \fg^{-1,1}_\varphi (U^{p,q}) \ := \ 
  \{ \xi(u) \ | \ \xi \in \fg^{-1,1}_\varphi \,,\ u \in U^{p,q} \} \,.
\end{equation}
In particular, given $\xi \in \fg^{-1,1}_\varphi$, we have 
\begin{equation}\label{E:x(U)}
  \xi(U^{p,q}) \ \subset \ U^{p-1,q+1} \,.
\end{equation}
The maps
\begin{subequations}\label{SE:psipq}
\begin{equation}
  \psi^{p,q}_\Omega : \fg^{-1,1}_\varphi \ \times \ U^{p,q} \ \to \ U^{p-1,q+1}
\end{equation}
sending
\begin{equation}
  (\xi,u) \ \mapsto \ \xi(u)
\end{equation}
\end{subequations}
are surjective.  Moreover, given fixed nonzero $u_0 \in U^{n,0}$, the map $\fg^{-1,1}_\varphi \to U^{n-1,1}$ sending $\xi \mapsto \xi(u_0)$ is an isomorphism.  It follows from the homogeneity of the bundles $\cF^p_\Omega$, and the $G_\bC$--equivariance of $\tau$, that $\tau$ is horizontal and of Calabi--Yau type.
\end{proof}

\subsection{Characteristic forms} \label{S:cfO}

In this section we describe the characteristic forms $\gamma^k_\Omega$ of \eqref{E:tau}.  The discussion will make use of results reviewed in the proof of Lemma \ref{L:gross}.

Since $\tau$ is $G_\bC$--equivariant and the bundles $\cF^p_\Omega \to \check D_\Omega$ are $\tAut(U_\bC,Q)$--homogeneous, we see that the push-forward $g_*: T_o \check \Omega \to T_{g \cdot o} \check \Omega$ is an isomorphism identifying $\mathbf{C}^k_{\tau,g\cdot o}$ with $\mathbf{C}^k_{\tau,o}$ for all $k$ and $g \in G_\bC$; that is, the characteristic forms of $\tau$ at $g \cdot o$ are isomorphic to those at $o$.  So it suffices to describe the characteristic forms at the point $o \in \Omega$.  It follows from $\cF^p_{\Omega,o}/\cF^{p+1}_{\Omega,o} = U^{p,n-p}$, the identification \eqref{E:TO}, and \eqref{E:x(U)} that $\gamma^k_{\Omega,o} : \tSym^kT_o\check\Omega \to \tHom(\cF^n_{\Omega,o} , \cF^{n-k}_{\Omega,o}/\cF^{n-k+1}_{\Omega,o})$ may be identified with the map
\begin{subequations}\label{SE:gammaO}
\begin{equation}
  \gamma^k_{\Omega,o} : \tSym^k\fg^{-1,1}_\varphi \ \to \ 
  \tHom(U^{n,0},U^{n-k,k}) 
\end{equation}
defined by 
\begin{equation}
  \gamma^k_{\Omega,o}(\xi_1\cdots\xi_k)(u) \ = \ 
  \xi_1 \cdots \xi_k(u)\,,
\end{equation}  
\end{subequations}
with $\xi_1 , \ldots , \xi_k \in \fg^{-1,1}_\varphi \subset \tEnd(U_\bC,Q)$ and $u \in U^{n,0}$.
 
\section{Proof of Theorem \ref{T:main1} (Main Theorem \ref{T:mt1})}

\subsection{The osculating filtration} \label{S:osc}

Let $X \inj \bP V_\bC$ be any complex submanifold.  The \emph{osculating filtration at $x \in X$} 
\[
  \cT_x^0 \ \subset \ \cT_x^1 \ \subset \cdots \subset \ \cT_x^m 
  \ \subset \ V_\bC
\]
is defined as follows.  First, $\cT_x^0 \subset V_\bC$ is the line parameterized by $x \in \bP V_\bC$.  Let $\widehat X \subset V_\bC\backslash\{0\}$ be the cone over $X$.  Let $\Delta = \{ z \in \bC \ : \ |z| < 1 \}$ denote the unit disc, and let $\cO(\Delta,0;\widehat X,x)$ denote the set of holomorphic maps $\a : \Delta \to \widehat X$ with $\a(0) \in \cT_x^0$.  Given one such curve, let $\a^{(k)}$ denote the $k$--th derivative $\td^k\a/\td z^k$.  Inductively,
\[
  \cT_x^k \ = \ \cT_x^{k-1} \ + \ 
  \tspan_\bC\{ \a^{(k)}(0) \ | \ \a \in \cO(\Delta,0;X,x) \} \,.
\]
Note that $\cT_x^1 = T_u \widehat X$ is the embedded tangent space at $u \in \cT_x^0$.  Here $m = m(x)$ is determined by $\cT_x^{m-1} \subsetneq \cT_x^m = \cT_x^{m+1}$.  

\subsection{Fundamental forms} \label{S:F}

If both $m$ and the rank of $\cT_x^k$ are independent of $x$, then the osculating filtrations define a a filtration $\cT^0_X \subset \cT^1_X \subset \cdots \subset \cT^m_X \subset X \times V_\bC$ of the trivial bundle over $X$.  Assume this is the case.  By construction the osculating filtration satisfies 
\begin{equation}\label{E:osc}
  \td\cT^k \ \subset \ \cT^{k+1} \ot \Omega^1_X
\end{equation}
Just as the IPR \eqref{E:IPR} lead to the characteristic forms \eqref{E:gammak}, the relation \eqref{E:osc} yields bundle maps
\[
  \psi^k_X : \tSym^k TX \ \to \ \tHom(\cT^0_X , \cT^k_X/\cT^{k-1}_X) \,,\quad k \ge 1 \,.
\]
This is the \emph{$k$--th fundamental form} of $X \inj \bP V_\bC$.  The image $\mathbf{F}^k_X \subset \tSym^k T^*X$ of the dual map is a vector subbundle of
\[
  \trank\,\mathbf{F}^k_X \ = \ \tdim\,\cT_x^k/\cT^{k-1}_x \,.
\]
Again, in mild abuse of terminology, we will call $\mathbf{F}^k_X$ the \emph{$k$--th fundamental forms of $X \subset \bP V_\bC$}.

Given two complex submanifolds $X,X' \inj \bP V_\bC$, we say that the fundamental forms of $X$ at $x$ are \emph{isomorphic} to those of $X'$ at $x'$ if there exists a linear isomorphism $\lambda : T_xX \to T_{x'}X'$ such that the induced linear map $\tSym\,T^*_{x'}X' \to \tSym\,T^*_xX$ identifies $\mathbf{F}^k_{X',x'}$ with $\mathbf{F}^k_{X,x}$.

Each $\mathbf{F}^k_{X,x}$ is a vector subspace of $\tSym^k T^*_xX$, and $d_{X,x}^k := \tdim_\bC\,\mathbf{F}^k_{X,x}$ is an example of an ``integer--valued differential invariant of $X \inj \bP V_\bC$ at $x$.''  Let
\[
  \mathbf{F}_{X,x} \ := \ 
  \bigoplus_{k\ge0} \mathbf{F}^k_{X ,x}\ \subset \ 
  \bigoplus_{k\ge0} \tSym^k T^*_xX \ =: \ 
  \tSym\,T^*_xX \,,
\]
and set $d_{X,x} := \tdim_\bC\,\mathbf{F}_{X,x} = \sum_{k\ge0} d^k_{X,x}$.  Regard $\mathbf{F}_{X,x}$ as an element of the Grassmannian $\tGr(d_{X,x},\tSym\,T^*_xX)$.  Note that $\tAut(T_xX)$ acts on this Grassmannian.  By \emph{integer--valued differential invariant of $X \inj \bP V_\bC$ at $x$} we mean the value at $\mathbf{F}_{X,x}$ of any $\tAut(T_xX)$--invariant integer-valued function on $\tGr(d_{X,x},\tSym\,T^*_xX)$.

A necessary condition for two fundamental forms $\mathbf{F}_{X,x}$ and $\mathbf{F}_{X',x'}$ to be isomorphic is that the integer-valued differential invariants at $x$ and $x'$, respectively, agree.

\begin{remark}
When $X \inj \bP V_\bC$ is a homogeneous embedding of a compact Hermitian symmetric space (such as the $\s : \check \Omega \inj \bP U_\bC$ of \eqref{E:sigma'}), there are only finitely many $\tAut(T_o\check\Omega)$--invariant integer-valued functions on $\tGr(d_{\s,o},\tSym\,T^*_o\check\Omega)$, and they distinguish/characterize the $\tAut(T_o\check\Omega)$--orbits \cite[Proposition 5]{MR2030098}.
\end{remark}

\subsection{Fundamental forms for $\s:\check\Omega\inj \bP U_\bC$}

Recall the maps $\sigma$ and $\tau$ of \eqref{E:sigma'} and \eqref{E:tau}, respectively.  Theorem \ref{T:HY} asserts that the Hermitian symmetric $\sigma(\check\Omega) \subset \bP U_\bC$ are characterized by their fundamental forms, up to the action of $\tAut(U_\bC)$.

\begin{theorem}[Hwang--Yamaguchi {\cite{MR2030098}}] \label{T:HY}
Assume that the compact dual $\check \Omega$ contains neither a projective space nor a quadric hypersurface as an irreducible factor.  Let $M \subset \bP U_\bC$ be any complex manifold, and let $x \in M$ be a point in a neighborhood of which all integer-valued differential invariants are constant.  If the fundamental forms of $M$ at $x$ are isomorphic to the fundamental forms of $\s : \check\Omega \inj \bP U_\bC$ at $o$, then $M$ is projective-linearly equivalent to an open subset of $\check\Omega$.
\end{theorem}

\begin{proposition} \label{P:C=F}
The $k$--th characteristic form $\gamma^k_\Omega$ of $\tau : \check\Omega \inj \check D_\Omega$ coincides with the the $k$--th fundamental form $\psi^k_\Omega$ of $\sigma : \check\Omega \inj \bP U_\bC$.
\end{proposition}

\begin{proof}
The proof is definition chasing.  Since both the Hodge bundles $\cF_\Omega^p$ and the osculating filtration $\cT_\Omega^k$ are homogeneous, and the maps $\s$ and $\tau$ are $G_\bC$--equivariant, it suffices to show that $\gamma^k_{\Omega,o} = \psi^k_{\Omega,o}$ at the point $o = P/P \in \check\Omega$.  The former is computed in \S\ref{S:cfO}; so it suffices to compute the latter and show that $\psi^k_{\Omega,o}$ agrees with \eqref{SE:gammaO}.  This follows directly from the the definition $\s(gP) = g \cdot \ell$ and the identifications \eqref{E:ell} and \eqref{E:TO}.
\end{proof}

\begin{remark}
A more detailed discussion of the fundamental forms of compact Hermitian symmetric spaces (such as $\check\Omega$) may be found in \cite[\S3]{MR2030098}
\end{remark}

\begin{corollary}\label{C:O-F=E}
The Hodge filtration $\left.\cF^p_\Omega\right|_{\tau(\check\Omega)}$ agrees with the osculating filtration $\cT^{n-p}_{\sigma(\check\Omega)}$.
\end{corollary}

\subsection{Characteristic versus fundamental forms} \label{S:CvF}

\begin{lemma}
Let $f: M \inj \check D$ be a CY map \emph{(\S\ref{S:CY})}.  Let $\pi : \check D \to \bP V_\bC$ be the projection of \S\ref{S:CY}.  Then $\cT^{n-k}_{\pi\circ f,x} \subset \cF^k_{f,x}$ for all $x \in M$.
\end{lemma}

\begin{proof}
This follows directly from the definitions of horizontality (\S\ref{S:prelim}) and Calabi--Yau type (\S\ref{S:CY}), and the osculating filtration (\S\ref{S:osc}).
\end{proof}

\begin{remark} \label{R:recover}
Let $f : M \inj \check D$ be a CY map, and recall the projection $\pi : \check D \to \bP V_\bC$ of \S\ref{S:CY}.  By definition $f(x) = \cF^\sb_{f,x}$.  So, if the Hodge and osculating filtrations agree, $\cF^k_{f,x} = \cT^{n-k}_{\pi\circ f,x}$, then we can recover $f$ from $\pi \circ f$.
\end{remark}

\begin{lemma} \label{L:C=F}
Let $f : M \inj \check D$ be a CY map.  If $\cT^{n-k}_{\pi\circ f, x} = \cF^k_{f,x}$ for all $x \in M$, then the characteristic and fundamental forms agree, $\mathbf{C}^{k}_f = \mathbf{F}^{n-k}_f$.
\end{lemma}

\begin{proof}
Again this is an immediate consequence of the definitions of the characteristic and fundamental forms (\S\ref{S:C} and \S\ref{S:F}, respectively).
\end{proof}

\begin{lemma} \label{L:Fisom}
Let $f : M \inj \check D_\Omega$ be a CY map.  Suppose that the characteristic forms $\mathbf{C}^\sb_f$ of $f$ are isomorphic to the characteristic forms $\mathbf{C}^\sb_\Omega$ of $\tau : \check\Omega \inj \check D_\Omega$.  Then the fundamental forms $\mathbf{F}^\sb_{\pi\circ f}$ and $\mathbf{F}^\sb_\s$ are isomorphic.
\end{lemma}

\begin{proof}
The lemma is a corollary of Corollary \ref{C:O-F=E} and Lemma \ref{L:C=F}.
\end{proof}

\begin{proof}[Proof of Theorem \ref{T:main1}]
First observe that we may reduce to the case that $\check \Omega$ is irreducible: for if $\check \Omega$ factors as $\check \Omega_1 \times \check \Omega_2$, then we have corresponding factorizations $\check D_\Omega = \check D_{\Omega_1} \times \check D_{\Omega_2}$ and $f = f_1 \times f_2$ with $f_i : M \to \check D_{\Omega_i}$; the theorem holds for $f$ if and only if it holds for the $f_i$.

Now suppose that $\check\Omega$ is a projective space.  Then $\check \Omega = \bP^1$.  In this case $\check \Omega = \check D_{\Omega}$, and the theorem is trivial.  Likewise if $\check\Omega$ is a quadric hypersurface, then $\check\Omega = \check D_\Omega$, and the theorem is trivial.  (In both these cases $\tau = \sigma$ and $\pi$ is the identity.)

The remainder of the theorem is essentially a corollary of Theorem \ref{T:HY} and Lemma \ref{L:Fisom}.  These results imply that there exists $g \in \tAut(U_\bC)$ so that $g \circ \pi \circ f(M)$ is an open subset of $\pi\circ\tau(\check\Omega) = \s(\check\Omega)$.  From Remark \ref{R:recover} we deduce that $g \circ f (M)$ is an open subset of $\tau(\check\Omega)$.
\end{proof}

\section{Main Theorem \ref{T:mt2}}

In this section we give a precise statement (Theorem \ref{T:main2}) and proof of Main Theorem \ref{T:mt2}.  The theorem assumes a stronger form of isomorphism between the characteristic forms of $\tau$ and $f$ than Main Theorem \ref{T:mt1}; specifically the identification $\mathbf{F}_\Omega \simeq \mathbf{F}_f$ will respect the polarization $Q$ in a way that is made precise by working on a natural frame bundle $\cE_Q \to \check D_\Omega$.

\subsection{The frame bundle $\cE_Q \to \check D_\Omega$}

Let $d+1 = \tdim\,U_\bC$, and let 
\[
  d^p + 1 \ := \ \dim\,F^p
\]
be the dimensions of the flags $(F^p)$ parameterized by $\check D_\Omega$.  Let $\cE_Q$ be the set of all bases $\mathbf{e} = \{ e_0 , \ldots , e_d\}$ of $U_\bC$ so that $Q(e_j,e_k) = \d^{d}_{j+k}$.  Note that we have bundle map
\[
  \begin{tikzcd}
    \cE_Q \arrow[d , -> , "\check\pi"] 
    \arrow[dd , ->, "\pi_\cQ"' , bend right] & \\
    \check D_\Omega \arrow[d , ->,"\pi"] & \\
    \cQ \arrow[r,phantom,":=",description] & 
    \{ [v] \in \bP U_\bC \ | \ Q(v,v) = 0 \}
  \end{tikzcd}
\]
given by
\begin{eqnarray*}
 \check\pi(\mathbf{e}) & = & (F^p) \,,
 \quad F^p = \tspan\{ e_0 , \ldots , e_{d^p} \} \,, \\
 \pi_\cQ(\mathbf{e}) & = & [e_0] \,.
\end{eqnarray*}

\subsection{Maurer--Cartan form}

The frame bundle $\cE_Q$ is naturally identified with the Lie group $\tAut(U_\bC,Q)$, 
\begin{equation}\label{E:Eisom}
  \cE_Q \ \simeq \ \tAut(U_\bC,Q) \,,
\end{equation}
and the bundle maps are equivariant with respect to the natural (left) action of $\tAut(U_\bC,Q)$.  Consequently, the (left-invariant) \emph{Maurer--Cartan form} on $\tAut(U_\bC,Q)$ defines a $\tAut(U_\bC,Q)$--invariant coframing $\theta = (\theta^k_j) \in \Omega^1(\cE_Q,\tEnd(U_\bC,Q))$.  Letting $e_j$ denote the natural map $\cE_Q \to U_\bC$, the coframing is determined by 
\begin{equation}\label{E:de}
  \td e_j \ = \ \theta_j^k \,e_k \,.
\end{equation}
(The `Einstein summation convention' is in effect throughout: if an index appears as both a subscript and a superscript, then it is summed over.  For example, the right-hand side of \eqref{E:de} should be read as $\sum_k \theta_j^k\,e_k$.)  The form $\theta$ can be used to characterize horizontal maps as follows: let $f : M \to \check D_\Omega$ be any holomorphic map and define 
\[
  \cE_f \ := \ f^*(\cE_Q)\,.
\] 
In a mild abuse of notation, we let $\theta$ denote both the Maurer-Cartan form on $\cE_Q$, and its pull-back to $\cE_f$.  Then it follows from the definition \eqref{E:IPR} that 
\begin{equation}\label{E:horiz1}
  \begin{array}{l}
  \hbox{the map $f$ is horizontal if and only if 
  $\left.\theta^\m_\n\right|_{\cE_f} = 0$ for all}\\
  \hbox{$d^{q+1}+1 \le \mu \le d^q$ and $d^{p+1}+1 \le \nu \le d^p$
  with $p-q  \ge 2$.}
  \end{array}
\end{equation}

\subsection{Precise statement of Main Theorem \ref{T:mt2}}

The precise statement (Theorem \ref{T:main2}) of Main Theorem \ref{T:mt2} is in terms of a decomposition of the Lie algebra $\tEnd(U_\bC,Q)$.  Recall the Hodge decomposition \eqref{SE:Uhd}, and define
\[
  E_\ell \ := \ \{ \xi \in \tEnd(U_\bC,Q) \ | \ \xi(U^{p,q}) \subset 
  U^{p+\ell,q-\ell} \} \,.
\]
Then 
\begin{equation}\label{E:Eell}
  \tEnd(U_\bC,Q) \ = \ \bigoplus_\ell E_\ell \,,
\end{equation}
and this direct sum is a graded decomposition in the sense that the Lie bracket satisfies
\begin{equation}\label{E:br2}
  [ E_k, E_\ell ] \ \subset \ E_{k+\ell} \,.
\end{equation}
Let $\theta_\ell \in \Omega^1(\cE_Q,E_\ell)$ denote the component of $\theta$ taking value in $E_\ell$.  It follows from \eqref{E:horiz1} that 
\begin{equation}\label{E:horiz2}
  \begin{array}{c}
  \hbox{a holomorphic map $f : M \to \check D_\Omega$ is horizontal}\\
  \hbox{if and only if 
  $\left.\theta_{-\ell}\right|_{\cE_f} = 0$ for all $\ell \ge 2$.}
  \end{array}
\end{equation}

Let $\tilde P \subset \tAut(U_\bC,Q)$ be the stabilizer of $\varphi = \tau (o) \in \check D$.  Notice that the fibre $\check\pi^{-1}(\varphi) \subset \cE_Q$ is isomorphic to $\tilde P$, and $\check \pi : \cE_Q \to \check D_\Omega$ is a principle $\tilde P$--bundle.  The Lie algebra of $\tilde P$ is 
\[
  E_{\ge0} \ := \ \bigoplus_{\ell\ge0 } E_\ell \,.
\]
Consequently, if $\theta = \theta_{\ge0} + \theta_-$ is the decomposition of $\theta$ into the components taking value in $E_{\ge0}$ and $E_- : = \op_{\ell>0}\,E_{-\ell}$, respectively, then
\begin{equation}\label{E:ker}
  \tker\,\check\pi_* \ = \ \tker\,\theta_{\ge0} \ \subset T\cE_Q \,.
\end{equation}

We may further refine the decomposition \eqref{E:Eell} by taking the representation \eqref{E:G-Aut(U,Q)} into account.  The latter allows us to view $\tEnd(U_\bC,Q)$ as a $G_\bC$--module via the adjoint action of $\tAut(U_\bC,Q)$ on the endomorphism algebra.  Likewise, we may regard $\fg_\bC$ as a subalgebra of $\tEnd(U_\bC,Q)$ via the induced representation $\fg \inj \tEnd(U,Q)$.  Since $\fg_\bC \subset \tEnd(U_\bC,Q)$ is a $G_\bC$--submodule and $G_\bC$ is reductive, there exists a $G_\bC$--module decomposition 
\[
  \tEnd(U_\bC,Q) \ = \ \fg_\bC \ \op \ \fg_\bC^\perp \,.
\]
Note that 
\begin{equation}\label{E:br1}
  [\fg_\bC , \fg_\bC ] \ \subset \ \fg_\bC \tand 
  [\fg_\bC , \fg_\bC^\perp] \ \subset \ \fg_\bC^\perp \,.
\end{equation}
where the Lie bracket is taken in $\tEnd(U_\bC,Q)$.  

Both $\fg_\bC$ and $\fg_\bC^\perp$ inherit graded decompositions
\begin{equation} \label{E:grd}
  \fg_\bC \ = \ \op\,\fg_\ell \tand \fg_\bC^\perp \ = \ \op\,\fg_\ell^\perp
\end{equation}
defined by $\fg_\ell := \fg_\bC \cap E_\ell$ and $\fg_\ell^\perp := \fg_\bC^\perp \cap E_\ell$.  From \eqref{E:br2} and \eqref{E:br1} we deduce
\begin{equation} \label{E:br3}
  [ \fg_k , \fg_\ell] \ \subset \ \fg_{k+\ell} \tand 
  [ \fg_k , \fg_\ell^\perp ] \ \subset \ \fg_{k+\ell}^\perp \,.
\end{equation}
Recall the Hodge decomposition \eqref{E:ghd} and note that $\fg_\ell = \fg^{\ell,-\ell}_\varphi$; in particular, $\fg_\ell = \{0\}$ if $|\ell| > 1$, so that 
\begin{equation}\label{E:gell}
  \fg_\bC \ = \ \fg_1 \,\op\,\fg_0 \,\op\, \fg_{-1} 
\end{equation}
and 
\begin{equation}\label{E:gperpell}
  \fg_\ell^\perp \ = \ E_\ell \quad \hbox{for all } \ |\ell| \ge 2 \,.
\end{equation}
Set 
\begin{eqnarray*}
  \fg_{\ge0} \ = \ \bigoplus_{\ell\ge0} \fg_\ell 
  &\hbox{and}& \fg_- \ = \ \bigoplus_{\ell < 0} \fg_\ell\,,\\
  \fg^\perp_{\ge0} \ = \ \bigoplus_{\ell\ge0} \fg^\perp_\ell 
  &\hbox{and}& \fg^\perp_- \ = \ \bigoplus_{\ell < 0} \fg^\perp_\ell \,.
\end{eqnarray*}
Let $\theta_{\fg_{\ge0}}$, $\theta_{\fg_{\ge0}^\perp}$, 
\[
  \w \ := \ \theta_{\fg_-} \tand \eta \ := \ \theta_{\fg_-^\perp}
\]
denote the components of $\theta$ taking value in $\fg_{\ge0}$, $\fg_{\ge0}^\perp$, $\fg_-$ and $\fg_-^\perp$, respectively.

Given any complex submanifold $\cM \subset \cE_Q$, we say that the restriction $\left.\w\right|_\cM$ is \emph{nondegenerate} if the linear map
\[
  \w : T_{\mathbf{e}} \cM \ \to \ \fg_-
\]
is onto for all $\mathbf{e} \in \cM$.  

\begin{example} \label{eg:tau}
Recall the horizontal, equivariant embedding $\tau : \check \Omega \to \check D_\Omega$.  It follows from \eqref{E:ker} and the fact that $\tau : \check \Omega \inj \check D_\Omega$ is $G_\bC$--equivariant that
\[
  \left. \eta \right|_{\cE_\tau} \ = \ 0
\]
and $\left.\w\right|_{\cE_\tau}$ is nondegenerate.
\end{example}

\noindent Our second main theorem asserts that these two properties suffice to characterize $\tau : \check \Omega \to \check D_\Omega$ up to the action of $\tAut(U_\bC,Q)$.

\begin{theorem} \label{T:main2}
Let $f : M \to \check D_\Omega$ be a horizontal map of Calabi--Yau type.  There exists $g \in \tAut(U_\bC,Q)$ so that $g \circ f(M)$ is an open subset of $\tau(\check \Omega)$ if and only if $\eta$ vanishes on $\cE_f$.
\end{theorem}

\noindent The theorem is proved in \S\ref{S:prf2}.

\subsection{Relationship to characteristic forms} \label{S:rel}

The purpose of this section is to describe the characteristic forms $\mathbf{C}^k_f$ when $\left.\eta\right|_{\cE_f} = 0$.  The precise statement is given by Proposition \ref{P:eta=0}.  It will be convenient to fix the following index ranges
\[
  d^{n-k+1}+1 \ \le \ \mu_k,\nu_k \ \le \ d^{n-k} \ \hbox{ with } k \ge 1 \,.
\]
As we will see below, the indices $1 \le \mu_1 , \nu_1 \le d^{n-1}$ are distinguished, and we will use the notation
\[
  1 \ \le \ a,b \ \le \ d^{n-1}
\]
for this range.  We claim that the equations
\begin{equation}\label{E:a0}
  \eta^a_0 \ = \ 0 \tand \theta^a_0 \ = \ \w^a_0 \,, \quad\hbox{for all } \ 
  1 \le a \le d^{n-1} 
\end{equation}
hold on $\cE_Q$.  (Note that the first implies the second, and visa versa.)  The way to see this is to observe that (i) $(\theta^a_0)_{a=1}^{d^{n-1}}$ is precisely the component of $\theta$ taking value in 
\[
  E_{-1} \,\cap\, \tHom(\cF^n_\varphi,\cF^{n-1}_\varphi) \ \simeq \ 
  \tHom(\cF^n_\varphi,\cF^{n-1}_\varphi/\cF^n_\varphi) \,,
\]
and (ii) the fact that $\tau$ is Calabi--Yau implies that the projection
\[
   T_o\check\Omega \,\simeq \, \fg_- \ \to \ 
  E_{-1} \,\cap\, \tHom(\cF^n_\varphi,\cF^{n-1}_\varphi/\cF^{n-1})
\]
is an isomorphism.  Therefore, 
\begin{equation}\label{E:isom}
  E_{-1} \,\cap\, \tHom(\cF^n_\varphi,\cF^{n-1}_\varphi) \ = \ 
  \fg_{-1} \,\cap\, \tHom(\cF^n_\varphi,\cF^{n-1}_\varphi) 
  \ \simeq \ \fg_{-1}\,.
\end{equation}
There are three important consequences of \eqref{E:isom}.  First, we have   
\[
  \theta^a_0 \ = \ \w^a_0 \,,
\]
which forces 
\[
  \eta^a_0 \ = \ 0 \,,
\]
for all $1 \le a \le d^{n-1}$.  Second, the fact that $\gamma_{f,x}$ is an isomorphism implies that $\left.\w\right|_{\cE_f}$ is nondegenerate.  Third, from $\fg_- = \fg_{-1}$ we conclude that
\[
  (\theta_\fg)^{\mu_k}_{\nu_\ell} \ = \ 0 \quad\hbox{when } \ k-\ell \ge 2 \,.
\]
It follows from \eqref{E:isom} that the remaining components of $\w=\theta_{\fg_-}$ may be expressed as 
\begin{equation}\label{E:r}
  \w^{\mu_k}_{\nu_{k-1}} \ = \ r^{\mu_k}_{\nu_{k-1}a}\,\w^a_0 \,,
\end{equation}
$k \ge 2$, for some holomorphic functions
\[
  r^{\mu_k}_{\nu_{k-1}a} : \cE_Q \ \to \ \bC \,.
\]
It will be convenient to extend the definition of $r^{\mu_k}_{\nu_{k-1}a}$ to $k=1$ by setting $r^a_{0b} := \d^a_b$.

\begin{proposition} \label{P:eta=0}
Let $f : M \to \check D_\Omega$ be a horizontal map of Calabi--Yau type.  Fix $\ell \ge 0$.  The component of $\theta$ taking value in 
\begin{equation}\label{E:rmketa}
  \fg^\perp_{-1} \ \bigcap \ 
  \bigoplus_{k\le\ell} \tHom(\cF^{n-k+1},\cF^{n-k})
\end{equation}
vanishes on $\cE_f$ if and only if the 
\[
  \tilde r^{\mu_{k}}_{a_k\cdots a_2 a_1} \ := \ r^{\mu_{k}}_{\nu_{k-1} a_k} \, 
  r^{\nu_{k-1}}_{\s_{k-2} a_{k-1}} \cdots r^{\tau_2}_{a_2 a_1}
\]
are the coefficients of $\gamma^k_f$ for all $k \le \ell$; that is,
\begin{equation} \label{E:eta=0}
  \gamma^{k}_{f,x}(\xi_k,\ldots,\xi_1) \ = \ 
  \left\{ e_0 \ \mapsto \ 
  \tilde r^{\m_{k}}_{a_k\cdots a_1} \,\w^{a_k}_0(\z_k) \cdots \w^{a_1}_0(\z_1) 
  \, e_{\m_{k}} 
  \quad\hbox{mod}\quad \cF_{f,x}^{n-k+1} \right\} \,,
\end{equation}
where $\z_i \in T_\mathbf{e} \cE'_f$ with $\mathbf{e} = \{ e_0 , \ldots , e_d\} \in \check\pi^{-1}(f(x))$ and $\check\pi_*(\z_i) = f_*(\xi_i)$.  In particular, $\left.\eta\right|_{\cE_f} = 0$ if and only if the characteristic forms are given by \eqref{E:eta=0} for all $k$.
\end{proposition}

\noindent Note that the component of $\theta$ taking value in \eqref{E:rmketa} is $(\eta^{\mu_\ell}_{\nu_{\ell-1}})_{\ell\le k}$.  The proposition is proved by induction in \S\S\ref{S:1}--\ref{S:k}; because the first nontrivial step in the induction is $\ell=3$, we work through the cases $\ell=1,2,3$ explicitly.

\begin{remark} \label{R:eta=0}
Suppose that $\mathbf{e} = \{ e_0 , \ldots , e_d \} \in \cE_{\tau,o}$.  Making use of \eqref{E:isom}, we may identify $\{e_1,\ldots,e_{d^{n-1}}\}$ with a basis of $\{\xi_1 , \ldots , \xi_{d^{n-1}}\}$ of $\fg_{-}$.  Then the coefficients $r^{\mu_k}_{\nu_{k-1}a}$ are determined by 
\begin{equation} \label{E:sol}
  \xi_a(e_{\nu_{k-1}}) \ = \ r^{\mu_k}_{\nu_{k-1}a} e_{\mu_k} 
  \quad\hbox{mod}\quad \cF^{n-k+1}_{\tau,o} \,.
\end{equation}
There are two important consequences of this expression:

(a) It follows from \eqref{SE:gammaO} that \eqref{E:eta=0} holds for $f = \tau$.

(b) Equation \eqref{E:isom} tells us that $\fg_{-1}$ is the graph over $E_{-1} \,\cap\,\tHom(\cF^{n}_\varphi , \cF^{n-1}_\varphi/\cF^{n}_\varphi)$ of a linear function
\[
  R : E_{-1} \,\cap\,\tHom(\cF^{n}_\varphi , \cF^{n-1}_\varphi/\cF^{n}_\varphi)
  \ \to \ \bigoplus_{k\ge1} \,  
  \tHom(\cF^{n-k}_\varphi , \cF^{n-k-1}_\varphi/\cF^{n-k}_\varphi) \,.
\]
The functions $r^{\mu_k}_{\nu_{k-1}a}(\mathbf{e})$ of \eqref{E:r} are the coefficients of this linear map with respect to the bases of $E_{-1} \,\cap\,\tHom(\cF^{n}_\varphi , \cF^{n-1}_\varphi/\cF^{n}_\varphi)$ and $\oplus_{k\ge1}\, \tHom(\cF^{n-k}_\varphi , \cF^{n-k-1}_\varphi/\cF^{n-k}_\varphi)$ determined by $\mathbf{e} \in \cE_Q$.  Assuming that \eqref{E:eta=0} holds, this implies that the $k$--th characteristic form of $f$ is isomorphic to that of $\tau$ in the following sense: given $\mathbf{e}_o \in \cE_\tau$ in the fibre over $o$ and $\mathbf{e}_x \in \cE_f$ in the fibre over $x$, there exists a unique $g \in \tAut(U_\bC,Q) \simeq \cE_Q$ so that $\mathbf{e}_x = g \cdot \mathbf{e}_o$.  The group element $g$ defines an explicit isomorphism between $\tSym^k T^*_o\check\Omega \ot \tHom(\cF^n_{\tau,o} , \cF^{n-k}_{\tau,o}/\cF^{n-k+1}_{\tau,o})$ and $\tSym^k T^*_x M \ot \tHom(\cF^n_{f,x} , \cF^{n-k}_{f,x}/\cF^{n-k+1}_{f,x})$ that identifies the $k$--th characteristic forms $\gamma^k_{\tau,o}$ and $\gamma^k_{f,x}$ at $o$ and $x$, respectively.  \emph{This is the precise sense in which the vanishing of $\eta$ on $\cE_f$ is a refined notion of agreement of the characteristic forms.}
\end{remark}

\begin{remark}\label{R:sol}
Recalling \eqref{E:ontoU}, and the identification $U^{p,q} = \cF^p_{\tau,o}/\cF^{p+1}_{\tau,o}$, \eqref{E:sol} implies that the system $\{r^{\mu_k}_{\nu_{k-1}a} Y_{\mu_k} = 0\}$ of $d^{n-1}(d^{k-1}-d^k)$ equations in the $d^{k}-d^{k+1}$ unknowns $\{Y_{\mu_k}\}$ has only the trivial solution $Y_{\mu_k}=0$.
\end{remark}

\subsubsection{The first characteristic form} \label{S:1}

Let $f : M \to \check D_\Omega$ be any horizontal map of Calabi--Yau type.  On the bundle $\cE_f$, \eqref{E:horiz1} and \eqref{E:a0} yield
\[
  \td e_0 \ = \ \theta^0_0 \, e_0 \ + \ \sum_{a = 1}^{d^{n-1}} \w^a_0 \,e_a \,.
\]
Consequently, the first characteristic form $\gamma_{f,x} : T_x M \to \tHom(\cF^n_{f,x} , \cF^{n-1}_{f,x}/\cF^n_{f,x} )$ is given by 
\begin{equation}\label{E:cf1}
  \gamma_{f,x}(\xi) \ = \ 
  \left\{ e_0 \ \mapsto \ \sum_{a = 1}^{d^{n-1}} \w^a_0(\z) \,e_a 
  \quad\hbox{mod} \quad e_0 \right\} \,,
\end{equation}
where $\z \in T_\mathbf{e} \cE_f$ with $\mathbf{e} = \{ e_0 , \ldots , e_d\} \in \check\pi^{-1}(f(x))$ and $\check\pi_*(\z) = f_*(\xi)$.  

This establishes Proposition \ref{P:eta=0} for the trivial case that $\ell=1$.

\subsubsection{The second characteristic form} \label{S:2}

From \eqref{E:horiz1} we see that 
\begin{equation}\label{E:m20}
  \theta^{\m_2}_0 = 0 \quad\hbox{on}\quad \cE_f
\end{equation}
for all $d^{n-1}+1 \le \mu_2 \le d^{n-2}$.  The derivative of this expression is given by the \emph{Maurer--Cartan equation}
\begin{equation} \label{E:mce}
  \td \theta \ = \ -\half [\theta,\theta] \,;\footnote{Given two Lie algebra valued $1$-forms $\phi$ and $\psi$, the Lie algebra valued 2-form $[\phi,\psi]$ is defined by $[\phi,\psi](u,v) := \half([\phi(u),\psi(v)] - [\phi(v),\psi(u)]$.} \quad
  \hbox{equivalently, } \ 
  \td \theta^j_k \ = \ - \theta^j_\ell \wedge \theta^\ell_k \,.
\end{equation}
Differentiating \eqref{E:m20} and applying \eqref{E:horiz1} yields
\[
  0 \ = \ \td\theta^{\m_2}_0 \ = \ -\theta^{\m_2}_a \wedge \w^a_0 
\]
on $\cE_f$.  Cartan's Lemma \cite{MR2003610} asserts that there exist holomorphic functions 
\[
  q^{\m_2}_{ab} \, = \, q^{\m_2}_{ba} : \cE_f \ \to \ \bC
\]
so that 
\begin{equation} \label{E:cf2}
 \theta^{\m_2}_a \ = \ q^{\m_2}_{ab} \,\w^b_0 \,.
\end{equation}
The $q^{\m_2}_{ab}$ are the coefficients of the second characteristic form; specifically,
\begin{equation}\label{E:cf2'}
  \gamma^2_{f,x}(\xi_1,\xi_2) \ = \ 
  \left\{ e_0 \ \mapsto \ 
  q^{\m_2}_{ab} \,\w^a_0(\z_1) \w^b_0(\z_2) \, e_{\m_2} \quad\hbox{mod}\quad 
  \cF_{f,x}^{n-1} \right\} \,,
\end{equation}
where $\z_i \in T_\mathbf{e} \cE'_f$ with $\mathbf{e} = \{ e_0 , \ldots , e_d\} \in \check\pi^{-1}(f(x))$ and $\check\pi_*(\z_i) = f_*(\xi_i)$.  

\begin{remark}
From Example \ref{eg:tau}, \eqref{E:r} and \eqref{E:cf2} we see that $q^{\mu_2}_{ab} = r^{\mu_2}_{ab}$ on $\cE_\tau$.  
\end{remark}

Returning to the bundle $\cE_f$, notice that $(\eta^{\mu_2}_a)$ is precisely the component of $\theta$ taking value in  
\[
  \fg^\perp_{-1} \,\cap\, \tHom(\cF^{n-1}_\varphi,\cF^{n-2}_\varphi) \,.
\]
Comparing \eqref{E:r} and \eqref{E:cf2}, we see that this component vanishes if and only if $r^{\mu_2}_{ab} = q^{\mu_2}_{ab}$ on $\cE_f$.  Noting that $\tilde r^{\mu_2}_{ab} = r^{\mu_2}_{ab}$, this yields Proposition \ref{P:eta=0} for $\ell=2$.

\subsubsection{The third characteristic form} \label{S:3}

From \eqref{E:horiz1} we see that 
\begin{equation}\label{E:m30}
  \theta^{\m_3}_a = 0 \quad\hbox{on}\quad \cE_f
\end{equation}
for all $d^{n-2}+1 \le \mu_2 \le d^{n-3}$.  Applying \eqref{E:horiz1}, the Maurer-Cartan equation \eqref{E:mce}, and substituting \eqref{E:cf2}, we compute
\[
  0 \ = \ -\td \theta^{\mu_3}_a 
  \ = \ \theta^{\mu_3}_{\nu_2} \wedge \theta^{\nu_2}_a
  \ = \ \theta^{\mu_3}_{\nu_2} \wedge q^{\nu_2}_{ab}\,\w^b_0 \,.
\]
Again Cartan's Lemma implies there exist holomorphic functions $q^{\nu_3}_{abc} : \cE_f \to \bC$, fully symmetric in the subscripts $a,b,c$, so that 
\begin{equation}\label{E:cf3}
  q^{\nu_2}_{ab}\, \theta^{\mu_3}_{\nu_2} \ = \ q^{\mu_3}_{abc} \,\w^c_0 \,.
\end{equation}
These functions are the coefficients of the third characteristic form of $f$ in the sense that 
\begin{equation}\label{E:cf3'}
  \gamma^3_{f,x}(\xi_1,\xi_2,\xi_3) \ = \ 
  \left\{ e_0 \ \mapsto \ 
  q^{\m_3}_{abc} \,\w^a_0(\z_1) \w^b_0(\z_2) \w^c_0(\z_3) \, e_{\m_3} 
  \quad\hbox{mod}\quad \cF_{f,x}^{n-2} \right\} \,,
\end{equation}
where $\z_i \in T_\mathbf{e} \cE'_f$ with $\mathbf{e} = \{ e_0 , \ldots , e_d\} \in \check\pi^{-1}(f(x))$ and $\check\pi_*(\z_i) = f_*(\xi_i)$.  

To prove Proposition \ref{P:eta=0} for $\ell=3$, note that \S\ref{S:2} yields $q^{\mu_2}_{ab} = r^{\mu_2}_{ab}$.  Then we can solve \eqref{E:cf3} for $\theta^{\mu_3}_{\nu_2}$ (Remark \ref{R:sol}).  In particular, there exist $q^{\mu_3}_{\nu_2a}$ so that $\theta^{\mu_3}_{\nu_2} = q^{\mu_3}_{\nu_2 a} \,\w^a_0$.  The component of $\theta$ taking value in
\[
  \fg^\perp_{-1} \,\cap\, \tHom(\cF^{n-2}_\varphi,\cF^{n-3}_\varphi)
\]
vanishes (equivalently, $\eta^{\mu_3}_{\nu_2}=0$) if and only if these $q^{\mu_3}_{\nu_2a}$ are the $r^{\mu_3}_{\nu_2a}$ of \eqref{E:r}; equivalently, \eqref{E:eta=0} holds for $k=3$.  This is Proposition \ref{P:eta=0} for $\ell=3$.

\subsubsection{And so on} \label{S:k}

Assume that Proposition \ref{P:eta=0} holds for a fixed $\ell \ge 3$.  Then we have $\theta^{\mu_k}_{\nu_{k-1}} = \w^{\mu_k}_{\nu_{k-1}} = r^{\mu_k}_{\nu_{k-1}a} \w^a_0$ for all $k \le \ell$.   As in \S\S\ref{S:2}--\ref{S:3} we obtain the coefficients of the $(\ell+1)$--st characteristic form by differentiating $\theta^{\mu_{\ell+1}}_{\nu_{\ell-1}}=0$  and invoking Cartan's Lemma to obtain
\[
  r^{\s_\ell}_{\nu_{\ell-1}a} \, \theta^{\mu_{\ell+1}}_{\s_\ell} \ = \ 
  q^{\mu_{\ell+1}}_{\nu_{\ell-1}ab} \, \w^b_0 \,,
\]
for some holomorphic functions $q^{\mu_{\ell+1}}_{\nu_{\ell-1}ab} : \cE_f \to \bC$, symmetric in $a,b$.  Then Remark \ref{R:sol} implies that there exist $q^{\mu_{\ell+1}}_{\nu_\ell a} : \cE_f \to \bC$ so that 
\[
  \theta^{\mu_{\ell+1}}_{\nu_\ell} \ = \ 
  q^{\mu_{\ell+1}}_{\nu_\ell a} \, \w^a_0 \,.
\]
The $q^{\mu_{\ell+1}}_{a_\ell\cdots a_1 a_0} := q^{\mu_{\ell+1}}_{\nu_\ell a_\ell} \, r^{\nu_\ell}_{\s_{\ell-1} a_{\ell-1}} \cdots r^{\tau_2}_{a_1 a_0}$ are the coefficients of the $(\ell+1)$--st characteristic form of $f$ in the sense that 
\begin{equation} \label{E:cf'}
  \gamma^{\ell+1}_{f,x}(\xi_\ell,\ldots,\xi_0) \ = \ 
  \left\{ e_0 \ \mapsto \ q^{\m_{\ell+1}}_{a_\ell\cdots a_0} \,
  \w^{a_\ell}_0(\z_k) \cdots \w^{a_0}_0(\z_0) \, e_{\m_{\ell+1}} 
  \quad\hbox{mod}\quad \cF_{f,x}^{n-\ell} \right\} \,,
\end{equation}
where $\z_i \in T_\mathbf{e} \cE'_f$ with $\mathbf{e} = \{ e_0 , \ldots , e_d\} \in \check\pi^{-1}(f(x))$ and $\check\pi_*(\z_i) = f_*(\xi_i)$.  The component of $\theta$ taking value in
\[
  \fg^\perp_{-1} \,\cap\, \tHom(\cF^{n-\ell}_\varphi,\cF^{n-\ell-1}_\varphi)
\]
vanishes (equivalently, $\eta^{\mu_{\ell+1}}_{\nu_\ell}=0$), if and only if the $q^{\mu_{\ell+1}}_{\nu_\ell a}$ are the $r^{\mu_{\ell+1}}_{\nu_\ell a}$ of \eqref{E:r}; equivalently, \eqref{E:eta=0} holds for $k \le \ell+1$.

This establishes Proposition \ref{P:eta=0}.

\subsection{Proof of Theorem \ref{T:main2}} \label{S:prf2}

\begin{claim} \label{cl:cE'}
It suffices to show that $\cE_f$ admits a sub-bundle $\cE_f^{\prime}$ on which $\theta_{\fg^\perp}$ vanishes.
\end{claim}

\begin{example}[Subbundle $\cG \subset \cE_\tau$] \label{eg:cG}
The bundle $\cE_\tau \to \check\Omega$ admits a subbundle $\cG$ that is isomorphic to the image of $G_\bC$ in $\tAut(U_\bC,Q)$, and on which the entire component $\theta_{\fg^\perp}$ of $\theta$ taking value in $\fg^\perp$ vanishes.  To see this, fix a basis $\mathbf{e}_o = \{ e_0 , \ldots , e_d\}$ that is adapted to the Hodge decomposition \eqref{SE:Uhd} in the sense that $e_0$ spans $U^{n,0}$, $\{ e_1 , \ldots , e_{d_1}\}$ spans $U^{n-1,1}$, et cetera, so that $\{ e_{d_{q-1}+1} , \ldots , e_{d_q} \}$ spans $U^{n-q,q}$, for all $q$.  Then $\mathbf{e}_o \in\cE_\tau$, and  
\[
\begin{tikzcd}
  \cG \arrow[r,phantom,":=",description] \arrow[d , ->>] 
  & G \cdot \mathbf{e}_o \ \subset \ \cE_\tau \\
  \tau(\check\Omega)
\end{tikzcd}
\]
is a $G_\bC$--homogenous subbundle with the properties that 
\begin{equation}\label{E:gperpcG}
  \left.\theta_{\fg^\perp}\right|_\cG \ = \ 0 \,,
\end{equation}
(in particular, $\left.\eta\right|_\cG = 0$)
and $\left.\theta_{\fg}\right|_\cG$ is a coframing of $\cG$ (so that $\left.\w\right|_{\cG}$ is nondegenerate).  
\end{example}

\begin{proof}
Recalling \eqref{E:br1}, the Maurer--Cartan equation $\td \theta = -\half[\theta,\theta]$ implies that $\{ \theta_{\fg^\perp} = 0\}$ is a Frobenius system on $\cE_Q$.  Notice that the bundle $\cG \subset \cE_Q$ of Example \ref{eg:cG} is the maximal integral through $\mathbf{e}_o$.  Since $\theta$ is $\tAut(U_\bC,Q)$--invariant, it follows that the maximal integral manifolds of the Frobenius system are the $g\cdot \cG$, with $g \in \tAut(U_\bC,Q)$.  Therefore, $g \cdot \cE_f^{\prime} \subset \cG$ for some $g \in \tAut(U_\bC,Q)$. From the $\tAut(U_\bC,Q)$--equivariance of $\check\pi$ we conclude that $g \circ f(M) \subset \check \Omega$.
\end{proof}

We will show that $\cE_f$ admits a sub-bundle $\cE_f'$ on which $\theta_{\fg^\perp}$ vanishes by induction.  Given $\ell \ge -1$, suppose that $\cE_f$ admits a subbundle $\cE^\ell_f$ on which the form $\theta_{\fg^\perp_k}$ vanishes for all $k \le \ell$.  This inductive hypothesis holds for $\ell = -1$ with $\cE_f = \cE_f^{-1}$.  

\begin{claim} \label{cl:max}
A maximal such $\cE_f^\ell$ will have the property that the linear map
\[
  \theta_{\ge\ell+2} : \tker\,\w \,\subset\, T_\mathbf{e} \cE_f^\ell 
  \ \to \ E_{\ge\ell+2}
\]
is onto for all $\mathbf{e} \in \cE_f^\ell$.
\end{claim}

\begin{proof}
Recollect that $\cE_Q \to \check D_\Omega$ is a principal $\tilde P$--bundle.  Given $g \in \tilde P$, let 
\[
  R_g : \cE_Q \ \to \ \cE_Q
\]
denote the right action of $\tilde P$.  Set $\tilde P_{\ell+2} := \exp( E_{\ge \ell+2}) \subset \tilde P$.  Then 
\[
  \tilde\cE_f^\ell \ := \ \{ R_g \mathbf{e} \ | \ g \in \tilde P_{\ell+2} \,,\
   \mathbf{e} \in \cE_f^\ell \} \ \supset \ \cE_f^\ell
\]
is a bundle over $M$, and $\theta_{\ge\ell+2} : \tker\,\w \subset T_\mathbf{e} \tilde\cE_f^\ell \to E_{\ge\ell+2}$ onto by construction.  Additionally, $R_g^*\theta = \tAd_{g^{-1}}\theta$ implies that $\theta_{\fg^\perp_{\le\ell}}$ vanishes on $\tilde \cE_f^\ell$.
\end{proof}

\noindent Given $\cE^\ell_f$, which we assume to be maximal, we will show that $\cE^{\ell+1}_f \subset \cE^\ell_f$ exists.  This will complete the inductive argument establishing the existence of the bundle $\cE_f'$ in Claim \ref{cl:cE'}. 

\begin{claim} \label{cl:lambda}
There exists a holomorphic map $\lambda : \cE_f^\ell \to \tHom(\fg_- , \fg^\perp_{\ell + 1})
 = \fg^\perp \ot \fg_-^*$ so that 
\begin{equation} \label{E:lam1}
  \theta_{\fg^\perp_{\ell+1}} \ = \ \lambda(\w) \,.
\end{equation}
\end{claim}

\begin{proof}
Since $\theta_{\fg^\perp_\ell}$ vanishes on $\cE^\ell_f$, the exterior derivative $\td \theta_{\fg^\perp_\ell}$ must as well.  Making use of the Maurer-Cartan equation \eqref{E:mce} and the relations \eqref{E:br3} we compute 
\begin{equation} \label{E:dth}
  0 \ = \ \td \theta_{\fg^\perp_\ell} \ = \  
  - [ \theta_{\fg^\perp_{\ell+1}} \,,\, \w ]
\end{equation}
on $\cE_f^\ell$.  The claim will then follow from Cartan's Lemma \cite[Lemma A.1.9]{MR2003610} once we show that the natural map
\begin{equation}\label{E:lam2}
  \fg^\perp_{\ell+1} \ \to \ \fg^\perp_\ell \ot \fg_-^*
  \quad\hbox{is injective} \,.
\end{equation}
The map \eqref{E:lam2} fails to be injective if and only if 
\[
  \Gamma_{\ell+1} \ := \ 
  \{ \z \in \fg^\perp_{\ell+1} \ | \ [\xi,\z]=0 \ \forall \ \xi \in \fg_- \}
\]
is nontrivial.  The Jacobi identity implies that $\Gamma_{\ell+1}$ is a $\fg_0$--module.  Inductively define $\Gamma_m := \fg_+(\Gamma_{m-1}) \subset \fg^\perp_m$.  The Jacobi identity again implies that $\Gamma = \op_{m\ge\ell+1}\,\Gamma_m$ is a $\fg_\bC$--module.  

Let $\mathtt{E} \in \tEnd(U_\bC,Q)$ be the endomorphism acting on $E_m$ by the scalar $m$. (That is, \eqref{E:Eell} is the eigenspace decomposition for $\mathtt{E}$.)  Then $\mathtt{E} \subset \fg_\bC$ lies in the center of $\fg_0 = \fk_\bC$ \cite[Proposition 3.1.2]{MR2532439}.  As a nontrivial semisimple element of $\fg_\bC$, $\mathtt{E}$ will act on any nontrivial $\fg_\bC$--module by both positive and negative eigenvalues.  Since $\ell \ge -1$, we see that $\mathtt{E}$ acts on $\Gamma$ by only non-negative eigenvalues.  This forces $\Gamma = \Gamma_{\ell+1} = \Gamma_0$ and $[\fg_\bC,\Gamma] = 0$.  

A final application of the Jacobi identity implies that $\fg_\bC \op \Gamma$ is a subalgebra of $\tEnd(U_\bC,Q)$.  Since $\fg_\bC \subset \tEnd(U_\bC,Q)$ is a maximal proper subalgebra \cite[Theorem 1.5]{MR0049903_trans}, and $\fg_\bC \op \Gamma_0 \not= \tEnd(U_\bC,Q)$, it follows that $\Gamma = \Gamma_0 = 0$.
\end{proof}

\noindent So to complete our inductive argument establishing the existence of $\cE_f'$ it suffices to show that there exists a subbundle $\cE^{\ell+1}_f \subset \cE^\ell_f$ on which $\lambda$ vanishes.  

\begin{claim}\label{cl:ker}
The map $\lambda$ takes value in the \emph{kernel} of the Lie algebra cohomology \cite{MR0142696} differential 
\[
  \d^1 : \fg^\perp \ot \fg_-^* \ \to \ 
  \fg^\perp \ot \tw^2 \fg_-^*
\]
defined by
\[
  \d^1(\a)(\xi_1,\xi_2) \ := \ 
  [ \a(\xi_1) , \xi_2] \,-\, [\a(\xi_2) , \xi_1] \,,
\]
where $\a \in \fg^\perp \ot \fg_-^* = \tHom(\fg_-,\fg^\perp)$ and $\xi_i \in \fg_-$.
\end{claim}

\begin{proof}
Substituting \eqref{E:lam1} into \eqref{E:dth} yields $[\lambda(\w) , \w] = 0$.  The claim follows.   
\end{proof}

\begin{claim} \label{cl:im}
Suppose $\lambda$ takes value in the \emph{image} of the Lie algebra cohomology differential
\[
  \d^0 : \fg^\perp \ \to \ \fg^\perp \ot \fg_-^*
\]
defined by 
\[
  \d^0(\z)(\xi) \ := \ [\xi,\z]
\]
with $\z \in \fg^\perp$ and $\xi\in\fg_-$.  Then there exists a subbundle $\cE_f^{\ell+1} \subset \cE_f^\ell$ on which $\lambda$ vanishes.
\end{claim}

\begin{proof} 
Differentiating \eqref{E:lam1} yields
\begin{equation}\label{E:dlam}
  0 \ = \ 
  \half \sum_{a+b=\ell+1} [ \theta_a , \theta_b ]_{\fg^\perp} \ + \ 
  \td \lambda \wedge \w \ - \ \lambda( [\theta_{\fg_0} , \w] ) \,.
\end{equation}
Claim \ref{cl:max} implies that $\theta(Z) = \z$ determines a unique, holomorphic vector field $Z$ on $\cE_f^\ell$.  (At the point $\mathbf{e} \in \cE_f^\ell$, the vector field is given by $Z_\mathbf{e} = \left.\frac{\td}{\td t} R_{\exp(t\z)} \mathbf{e} \right|_{t=0}$.) Taking the interior product of $Z$ with \eqref{E:dlam} yields
\begin{equation}\label{E:lam2}
  0 \ = \ (Z\lambda)(\w) \ + \ [\z,\w] \,.
\end{equation}
That is, $Z\lambda = \td \lambda(Z) = \tad_\z$.  Given $\mathbf{e} \in \cE_{f,x}^\ell$, set $\lambda_t := \lambda_{\mathbf{e}(t)}$ with $\mathbf{e}(t) := R_{\exp(t\z)} \mathbf{e}$.  Then \eqref{E:lam2} implies we may solve $\lambda_t = 0$ for $t$ if and only if $\lambda_\mathbf{e}$ takes value in the image of $\d^0$.
\end{proof}

It follows from Claims \ref{cl:ker} and \ref{cl:im} that the bundle $\cE_f^{\ell+1}$ exists if the cohomology group
\[
  H^1(\fg_-,\fg^\perp) \ := \ \frac{\tker\,\d^1}{\tim\,\d^0}
\]
is trivial.  In general $H^1(\fg_-,\fg^\perp) \not=0$.  Happily it happens that we don't need all of $H^1(\fg_-,\fg^\perp)$ to vanish, just the positively graded component.  To be precise, the gradings \eqref{E:grd} induce a graded decomposition 
\[
  \fg^\perp \ot \fg_-^* \ = \ 
  \bigoplus_\ell \fg^\perp_\ell \ot \fg_-^* \,.
\]
Since $\fg_- = \fg_{-1}$, the dual $\fg_-^*$ has graded degree $1$.  Consequently, $\fg^\perp_\ell \ot \fg_-^*$ has graded degree $\ell+1$.  The Lie algebra cohomology differentials $\d^1$ and $\d^0$ preserve this bigrading, and so induce a graded decomposition of the cohomology
\[
  H^1(\fg_-,\fg^\perp) \ = \ \bigoplus_\ell H^1_\ell
\]
where the component of graded degree $\ell+1$ is 
\[
  H^1_{\ell+1} \ := \ 
  \frac{\tker\,\{\d^1 : \fg^\perp_\ell \ot \fg_-^* \ \to \ 
  			\fg^\perp_{\ell-1} \ot \tw^2 \fg_-^*\}}
  {\tim\,\{\d^0 : \fg^\perp_{\ell+1} \ \to \ \fg^\perp_\ell \to \fg_-^* \}} \,.
\]
Since $\lambda$ takes value in $\fg^\perp_{\ell+1} \ot \fg_-^*$, and the latter is of pure graded degree $\ell+2 \ge 1$.  Consequently, 
\begin{equation}\label{E:cE'}
\begin{array}{c}
\hbox{\emph{there exists a subbundle $\cE_f'$ of $\cE_f$ on which}}\\
\hbox{\emph{$\theta_{\fg^\perp}$ vanishes if $H^1_m = 0$ for all $m \ge 1$.}}
\end{array}
\end{equation}
To complete the proof of Theorem \ref{T:main2} we make the following observations: First, as in the proof of Theorem \ref{T:main1} we may reduce to the case that $\check \Omega$ is irreducible.  Also as in that proof, the case that $\check\Omega$ is either a projective space (necessarily $\bP^1$) or a quadric hypersurface is trivial.

In the remaining cases $H^1_m = 0$ for all $m \ge 1$; this is a consequence of Kostant's theorem \cite{MR0142696} on Lie algebra cohomology; see \cite[Proposition 7]{MR2030098} or \cite[\S7.3]{MR3004278}.  The theorem now follows from Claim \ref{cl:cE'} and \eqref{E:cE'}.

\def\cprime{$'$} \def\Dbar{\leavevmode\lower.6ex\hbox to 0pt{\hskip-.23ex
  \accent"16\hss}D}

\end{document}